\documentclass[12pt]{amsart}

\usepackage{graphicx}
\usepackage{pdfsync}
\usepackage{a4,fullpage}
\usepackage{enumerate}
\usepackage{latexsym}
\usepackage{amsfonts}
\usepackage[utf8]{inputenc}
\usepackage{amsmath}
\usepackage{amssymb}

\usepackage{amscd}
\usepackage[all]{xy} 
\usepackage{mathrsfs} 
\usepackage{stmaryrd} 
\usepackage{amsopn}
\usepackage{eepic}
\usepackage{epic}
\usepackage{eucal}

\usepackage{epsfig}
\usepackage{psfrag}
\usepackage[justification=centering]{caption}

\newtheorem{theorem}{Theorem}[section]
\newtheorem{prop}[theorem]{Proposition}
\newtheorem{lemma}[theorem]{Lemma}

\theoremstyle{remark}
\newtheorem{remark}[theorem]{Remark}

\theoremstyle{definition}
\newtheorem{definition}[theorem]{Definition}
\newtheorem{example}[theorem]{Example}
\newtheorem{problem}[theorem]{Problem}
\newtheorem*{ack}{Acknowledgements}

\newcommand{\R}{{\mathbb R}}
\newcommand{\Z}{{\mathbb Z}}
\newcommand{\Q}{{\mathbb Q}}
\newcommand{\N}{{\mathbb N}}
\newcommand{\T}{{\mathbb T}}
\newcommand{\la}{\lambda} 
\newcommand{\om}{\omega} 
\newcommand{\Mo}{\mathcal{M}_\mathbb{T}}
\newcommand{\Mt}{\widetilde{\mathcal{M}_\mathbb{T}}}
\newcommand{\Dtt}{\widetilde{\mathcal{D}_\mathbb{T}}}
\newcommand{\Dt}{\mathcal{D}_\mathbb{T}}

\DeclareMathOperator{\len}{length}
\DeclareMathOperator{\symp}{Sympl}
\DeclareMathOperator{\agl}{AGL}

\begin{document}

\title{{\bf Moduli spaces of toric manifolds}}

\author{\'A.\ Pelayo}
\address{Mathematics Department, Washington University, One Brookings Drive, St.\ Louis, MO 63130-4899, USA.}
\curraddr{School of Mathematics, Institute for Advanced Study, Einstein Drive, Princeton,  NJ 08540 USA.}
\email{apelayo@math.wustl.edu}
\thanks{AP was partly supported by NSF Grants
DMS-0965738 and DMS-0635607, an NSF CAREER Award, a Leibniz Fellowship,
Spanish Ministry of Science Grant MTM 2010-21186-C02-01, 
and by the Spanish National Research Council. }

\author{A.\ R.\ Pires}
\address{Department of Mathematics, Cornell University, 310 Malott Hall, Ithaca, NY, 14853-4201, USA.}
\email{apires@math.cornell.edu}
\thanks{ARP was partly supported by an AMS-Simons Travel Grant.}

\author{T.\ S.\ Ratiu}
\address{Section de Math\'ematiques and Bernoulli Center, Ecole Polytechnique F\'ed\'erale de Lausanne, Station 8, Lausanne, 1015, Switzerland.}
\email{tudor.ratiu@epfl.ch}
\thanks{TSR was partly supported by a MSRI membership, Swiss NSF grant 200021-140238, a visiting position at IHES, and by the government grant of the Russian Federation for support of research projects implemented by leading scientists, Lomonosov Moscow State University under the agreement No.\ 11.G34.31.0054.}

\author{S.\ Sabatini}
\address{Section de Math\'ematiques, Ecole Polytechnique F\'ed\'erale de Lausanne, Station 8, Lausanne, 1015, Switzerland.}
\email{silvia.sabatini@epfl.ch}

\date{March 26, 2013}

\begin{abstract}
We construct a distance on the moduli space of 
symplectic toric manifolds of dimension four. 
Then we study some basic topological
properties of this space, in particular, path-connectedness, 
compactness, and completeness.
The construction of the distance is related to the 
Duistermaat-Heckman measure and the Hausdorff metric.
While the moduli space,  its topology and metric, may be 
constructed in any dimension, the tools we use in the proofs 
are four-dimensional, and hence so is our main result.
\end{abstract}  



\maketitle

\section{Introduction} \label{sec:intro}

A \emph{toric integrable system}  
$\mu:=(\mu_1,\ldots,\mu_n) \colon M \to \mathbb{R}^n$
is an integrable system on a connected symplectic 
$2n$-dimensional manifold $(M,\omega)$ in which all the 
flows generated by the $\mu_i$, $i=1,\ldots,n$, are periodic 
of a fixed period. That is, there is a Hamiltonian action 
of a torus $\T$ of dimension $n$ on $M$ with momentum map 
$\mu$. We will assume that this action is effective and 
that $M$ is compact. In this case, the quadruple 
$(M,\omega,\mathbb{T},\mu)$ is often called a 
\emph{symplectic toric manifold} of dimension
$2n$, to emphasize the connection with toric varieties (in 
fact, all symplectic toric manifolds are toric
varieties, e.g., see Remark \ref{De} and 
\cite{De1988,DuPe2009}). 

The goal of the paper is to construct natural topologies on 
moduli spaces of compact symplectic toric $4$-manifolds under
natural equivalence relations and study some of their basic 
topological properties. 

Throughout most of this paper, we will assume that $n=2$, 
but several definitions and statements hold in more 
generality.

\subsection{Conventions} \label{conventions}

Let, throughout this paper, $\mathbb{T}:=\mathbb{T}^n$ denote 
the $n$-dimensional standard torus
$$
\mathbb{T}^n=\overbrace{\mathbb{T}^1\times \cdots \times \mathbb{T}^1}^{n\;\;\text{times}}\;,
$$
i.e., the Cartesian product of $n$ copies of the
circle $\mathbb{T}^1$ equipped with the product  operation.  
Denote by $\mathfrak{t}$ the Lie algebra $\mbox{Lie}(\T)$ 
of $\T$ and by $\mathfrak{t}^*$ the dual of $\mathfrak{t}$.  
Strictly speaking, the momentum map $\mu$ of the Hamiltonian 
action of $\mathbb{T}$ on a manifold $M$ is a map $M \to 
\mathfrak{t}^*$.  However, the presentation is simpler, if 
from the beginning we consider this map as a map 
$\mu \colon M \to \mathbb{R}^n$. How to do this is a standard, 
but not canonical, procedure. Choose an epimorphism 
$E \colon \mathbb{R} \to \mathbb{T}^1$, for instance, 
$x \mapsto e^{2\pi\, \sqrt{-1}x}$. This Lie group epimorphism 
has discrete center $\mathbb{Z}$ and the inverse of the  
corresponding Lie algebra isomorphism is given by
$\mbox{Lie}(\mathbb{T}^1) \ni \frac{\partial}{\partial x} 
\mapsto \frac{1}{2\pi}\in \mathbb{R}$. Thus, for
$\mathbb{T}^n = \left(\mathbb{T}^1\right)^n$, we get the 
non-canonical isomorphism between the corresponding
commutative Lie algebras
\[
\mbox{Lie}(\mathbb{T}^n) =
\mathfrak{t} \ni \frac{\partial}{\partial x_k} \longmapsto 
\frac{1}{2\pi} \, {\rm e}_k \in\mathbb{R}^n,
\]  
where ${\rm e}_k$ is the $k^{\rm th}$ element in the canonical basis
of $\mathbb{R}^n$. Choosing an inner product 
$\langle \cdot,\cdot \rangle$ on $\mathfrak{t}$, we obtain an 
isomorphism $\mathfrak{t} \to \mathfrak{t}^*$, and hence 
taking its inverse and composing it with the isomorphism
$\mathfrak{t} \to \mathbb{R}^n$ described above, we get an 
isomorphism $\mathcal{I}: \mathfrak{t}^* \to \mathbb{R}^n$. 
In this way, we obtain a momentum map $\mu=\mu_{\mathcal{I}} \colon M \to \mathbb{R}^n$.

If $(M, \omega)$ is a symplectic manifold, denote by 
$\operatorname{Sympl}(M)$ the group of symplectic diffeomorphisms of $M$.

\subsection{The moduli space $\Mo$}

With the conventions in Section \ref{conventions}, where 
$\T$ and the identification $\mathcal{I} \colon
\mathfrak{t}^* \to \mathbb{R}^n$ are \emph{fixed}, we next 
define the moduli space of toric manifolds.    
Let $(M,\omega,\T,\mu \colon M \to \mathbb{R}^n)$ and 
$(M',\omega',\T,\mu' \colon M \to \mathbb{R}^n)$ be symplectic 
toric manifolds, with effective symplectic
actions $\rho\colon \T\to \symp (M,\omega)$ and 
$\rho'\colon \T\to \symp (M',\omega')$. 
These two
symplectic toric manifolds are \emph{isomorphic} if 
there exists an equivariant symplectomorphism  
$\varphi\colon M\to M'$ (i.e., $\varphi$ is a diffeomorphism 
satisfying $\varphi^\ast \omega' = \omega$ which intertwines the $\T$ actions) such that 
$\mu'\circ \varphi=\mu$ (see also \cite[Definition I.1.16]{AuCaLe}). We denote by $\Mo:=\Mo^{\mathcal{I}}$ 
the \emph{moduli space} (the set of equivalence classes) of 
$2n$-dimensional symplectic toric manifolds under this 
equivalence relation. 
The motivation for introducing this moduli space
comes from the following seminal result, due to Delzant \cite{De1988}.

\begin{theorem}[{\cite[Theorem 2.1]{De1988}}]\label{Del1}
Let $(M,\omega,\T,\mu)$ and $(M',\omega',\T,\mu' )$ be two toric symplectic
manifolds. If $\mu(M)=\mu'(M')$ then there exists an equivariant
symplectomorphism $\varphi\colon (M,\omega)\to (M',\omega')$ such that the following
diagram
\begin{equation*}
    \xymatrix{ (M,\om) \ar[d]_-{\mu} \ar[r]^-{\varphi} & (M',\om') \ar[d]^-{\mu'} \\
      \mu(M) \ar[r]^-{\text{Id}} & \mu'(M')}
  \end{equation*}
commutes.
\end{theorem}

The convexity theorem of Atiyah \cite{At1982} and Guillemin\--Sternberg \cite{GuSt1982} asserts that
the image of the momentum map is a convex polytope. In addition, if the action is toric (the acting torus is precisely half the dimension of the manifold) the momentum image is a Delzant polytope (see Section \ref{sec:dp}). 
Let $\mathcal{D}_\mathbb{T}$ denote the set of Delzant
polytopes. 
As a consequence of Theorem \ref{Del1}, the following map
\begin{equation}
\label{map}
\left[ (M,\omega,\T,\mu) \right] \ni \mathcal{M}_\mathbb{T}
\longmapsto \mu(M) \in \mathcal{D}_\mathbb{T},
\end{equation}
is an injection. However, Delzant also shows how from a Delzant
polytope it is possible to reconstruct a symplectic toric manifold, thus implying that
\eqref{map} is a \emph{bijection}.

To simplify notations, we usually write 
$(M,\omega,\T,\mu)$ identifying the representative with
its equivalence class $\left[(M,\omega,\T,\mu) \right]$
in $\mathcal{M}_\mathbb{T}$.

\begin{remark}
If we choose a different identification 
$\mathcal{I}' \colon \mathfrak{t}^* \to \mathbb{R}^n$ in Section \ref{conventions},
then the resulting moduli space $\Mo^{\mathcal{I}'}$ is, in general, a set \emph{different} from $\mathcal{M}_{\T}$. 
However, $\Mo$ and $\Mo^{\mathcal{I}'}$ are in bijective correspondence 
by a map which preserves the structures that this paper deals with (Section \ref{top}). 
An alternative and equivalent approach to this convention is to
define $\mathcal{M}_{\mathbb{T}}$ to be a space of pairs, where the second element
of the pair involves the Lie algebra identification.  We shall use the first convention throughout the
paper (see Section \ref{stm}).

\end{remark}

\begin{remark}
Note that for every non-zero $c\in \R^n$ the equivalence
class of $(M,\omega,\T,\mu)$ is different from that of $(M,\omega,\T,\mu+c)$.
We distinguish these two spaces since for general Hamiltonian $G$-actions the constant
$c$ is an important element of $\mathfrak{g}^*$. 
Indeed, given a connected Lie group $G$ and a Hamiltonian
$G$-space $(M,\omega,G)$ with moment map $\mu\colon M\to \mathfrak{g}$,
if $\mu'\colon M\to \mathfrak{g}$ is a different choice of moment map for the $G$-action then
$\mu-\mu'=c\in[\mathfrak{g},\mathfrak{g}]^0\subset \mathfrak{g}^*$, the annihilator of the commutator
ideal of $\mathfrak{g}$, which coincides with $H^1(\mathfrak{g};\R)$, the first
Lie algebra cohomology group (see for example \cite[$\S 26.2$]{anacannas}).  
\end{remark}

\subsection{The moduli space $\Mt$}

Following \cite{KKP}, we say that two symplectic toric 
manifolds $(M,\omega,\T,\mu)$ and $(M',\omega',\T,\mu')$ are 
\emph{weakly isomorphic} if there exists an automorphism of 
the torus $h\colon \T\to \T$ and an $h$-equivariant 
symplectomorphism $\varphi\colon M\to M'$, i.e.,  the 
following diagram commutes:
\begin{equation}\label{comm action1}
\xymatrix{
\T\times M \;\; \ar[r]^{\;\;\rho^*} \ar[d]_{(h,\varphi)} & M 
\ar[d]_{\varphi} \\
\T \times M' \;\;\ar[r]^{\;\;\rho'^*} &  M',
 } 
\end{equation}
where $\rho^*(x,m): = \rho(x)(m)$ for all $x \in\mathbb{T}$,
$m \in M$, and similarly for $\rho'^*$.

We denote by $\Mt$ the \emph{moduli space of weakly 
isomorphic $2n$-dimensional symplectic toric manifolds} 
(see Section \ref{stm} for more details).

Two weakly isomorphic toric manifolds $(M,\omega,\T,\mu)$ and 
$(M',\omega',\T,\mu')$ are isomorphic if and only if $h$ in 
(\ref{comm action1}) is the identity and $\mu'=\mu \circ 
\varphi$.

\subsection{Topologies and metrics} \label{top}

We consider the space of Delzant polytopes $\Dt$ and turn it 
into a metric space by endowing it with the distance function 
given by the volume of the symmetric difference 
$$
(\Delta_1\smallsetminus \Delta_2) \cup 
(\Delta_2\smallsetminus \Delta_1)
$$ 
of  any two polytopes.

The map (\ref{map}) allows us to define a metric $d_\T$ on 
$\Mo$ as the pullback of the metric defined on $\Dt$, thereby 
getting the metric space $(\Mo,d_\T)$.  \emph{This metric 
induces a topology $\nu$ on $\Mo$} (so, by definition,  it
follows that $(\Mo,\nu)$ is a metrizable topological space).

Let $\agl(n,\Z):=\operatorname{GL}(n,\Z)\ltimes \R^n$ be the 
group of affine transformations of $\R^n$ given by 
$$
\mathbb{R}^n\ni x\mapsto Ax+c \in\mathbb{R}^n,
$$ 
where $A\in \operatorname{GL}(n,\Z)$ and $c\in \R^n$.
We say that two Delzant polytopes $\Delta_1$ and $\Delta_2$
are \emph{$\agl(n,\Z)$-equivalent} if there exists 
$\alpha\in \agl(n,\Z)$ such that $\alpha(\Delta_1)=\Delta_2$.
Let $\Dtt$ be the moduli space of Delzant polytopes relative
to $\agl(n,\Z)$-equivalence; we endow this space with the 
quotient topology induced by the projection map
$$
\pi\colon \Dt \to \Dtt\simeq \Dt/\agl(n,\Z).
$$

As we shall see in Section \ref{stm}, there exists a bijection 
$\Psi:\Mt\rightarrow \Dtt$ (in fact, it is induced by 
(\ref{map})); thus $\Mt$ is also a topological space, with 
topology $\widetilde{\nu}$ induced by $\Psi$. \emph{We denote 
this topological space by $(\Mt,\widetilde{\nu})$.}

\subsection{Main Theorem}

Let $\mathfrak{B}(\R^n)$ be the $\sigma$-algebra of Borel sets 
of $\R^n$, the map 
$$\lambda\colon\mathfrak{B}(\R^n)\to\R_{\geq 0} 
\cup \{\infty\}$$ 
be the Lebesgue measure on $\R^n$, and 
$\mathfrak{B}'(\R^n)\subset \mathfrak{B}(\R^n)$
the subset of Borel sets with finite Lebesgue measure. Define  
\begin{eqnarray} \label{first}
 d(A,B):=\left\|\chi_A-\chi_B\right\|_{{\rm L}^1},
\end{eqnarray}
where $\chi_C\colon \R^n\to \R$ denotes the characteristic function of $C\in \mathfrak{B}'(\R^n)$. 
This extends the distance function defined above on $\Dt$, but it is \emph{not} a metric on $\mathfrak{B}'(\R^n)$. 
Identifying the sets $A, B \in \mathfrak{B}'(\R^n)$ for which
$d(A,B)=0$, we obtain a metric on the resulting quotient space of $\mathfrak{B}'(\R^n)$ (see Section
\ref{subsection:metric} for details). 

Let $\mathcal{C}$ be the space of convex compact subsets of 
$\R^2$ with positive Lebesgue measure,
$\varnothing$ the empty set, and
$$\hat{\mathcal{C}}:=\mathcal{C}\cup\{\varnothing\}.$$ 
Then $\hat{\mathcal{C}}$ equipped with the distance function 
$d$ in (\ref{first}) is a metric space. 
 
 We prove the following theorem.

\begin{theorem} \label{key}
Let $\Mo$ and $\Mt$ be the moduli spaces of toric 
four-dimensional manifolds, under isomorphisms and 
equivariant isomorphisms, respectively.  Then:
\begin{itemize}
\item[{\rm (a)}] $(\Mt,\widetilde{\nu})$  is path-connected; 
\item[{\rm (b)}] $(\Mo,d_\T)$ is neither locally compact nor 
a complete metric space. Its completion can be identified with 
the metric space $(\hat{\mathcal{C}},d)$ in the following 
sense: identifying $(\Mo,d_\T)$ with $(\Dt,d)$ via 
{\rm (\ref{map})}, the completion of  $(\mathcal{D}_{\T},d)$ 
is $(\hat{\mathcal{C}},d)$.  
\end{itemize}
\end{theorem}

\begin{remark}
Metric spaces are Tychonoff (that is, completely regular
and Hausdorff), therefore $\Mo$ is Tychonoff. The 
Stone-\v{C}ech compactification \cite{St1937,Ce1937}  
can be applied to Tychonoff spaces. The Stone-\v{C}ech 
compactification, in general, gives rise to a compactified 
space  which is Hausdorff and normal. Hence $\Mo$ admits a 
Hausdorff compactification.
\end{remark}

\begin{remark}
Theorem \ref{key} positively answers  the case $2n=4$ of 
Problem 2.42 in \cite{PeVN2012}. We do not know if the 
analogous statement to Theorem \ref{key} holds in dimensions 
greater than or equal to six. Note that the constructions of 
the moduli spaces $\Mo$ and $\Mt$  do not depend on dimension.
\end{remark}

\vspace{3mm}

\subsection*{Structure of the paper}
In Section \ref{sec:dp} we introduce the topological spaces 
we are going to work with, involving Delzant polytopes and 
symplectic toric manifolds, under certain 
equivalence relations. The ingredient that allows us to relate 
these two categories of spaces is the 
Delzant classification theorem (Theorem \ref{theo:delzant}).

Section \ref{sec:conn} starts with a detailed analysis
of how to construct Delzant polygons (i.e., polytopes 
of dimension $2$) following a simple recursive procedure 
presented in \cite{KKP}. This recipe is a main technical tool 
for the present paper; no such method is known for polytopes 
of dimension greater than or equal to three. The remainder 
of this section and Section \ref{sec:compl} are devoted to 
proving the path-connectedness and metric properties of the
space of Delzant polygons  (or rather, a natural quotient of 
it); the main theorem of the paper is implied by the results 
proven in these sections. 

Finally, in Section \ref{sec:problems}, several open problems 
are presented.  Appendix \ref{sec:appendix} contains
a brief review of the polytope terms and results we use in 
the paper.

\begin{ack}
We would like to thank the anonymous referee who made many useful comments
and clarifications which have significantly improved an earlier version of the paper.
AP is grateful to Helmut Hofer for discussions and support. He also thanks Isabella Novik for 
discussions concerning general polytope theory, and Problem \ref{p30}, during a visit to 
the University of Washington in 2010.  The authors are also grateful to Victor Guillemin and
Allen Knutson for helpful advice. 
\end{ack}

\section{Delzant polytopes and toric  manifolds} 
\label{sec:dp}

\subsection{A metric on the space of Delzant polytopes 
$\mathcal{D}_\T$}
\label{subsection:metric}

In this paper we are interested only in convex full 
dimensional polytopes, which we will simply call polytopes.
We refer to Appendix \ref{sec:appendix} for the basic 
terminology and results on polytopes.

\begin{definition}
\label{del 2} 
(following \cite{anacannas})
A convex polytope $\Delta$ in $\R^n$ is a \emph{Delzant 
polytope} if it is simple, rational and smooth:
\begin{itemize}
\item[{\rm (i)}] $\Delta$ is \emph{simple} if there are 
exactly $n$ edges meeting at each vertex $v\in V$;
\item[{\rm (ii)}] $\Delta$ is \emph{rational} if for every 
vertex $v\in V$, the edges meeting at $v$ are of the form 
$v+t u_i$, where $t\geq 0$ and $u_i\in \Z^n$;
\item[{\rm (iii)}] A vertex $v\in V$ is \emph{smooth} if the 
edges meeting at $v$ are of the form $v+tu_i$, $t\geq 0$, 
where the vectors $u_1,\ldots,u_n$ can be chosen to be a 
$\Z$ basis of $\Z^n$. $\Delta$ is \emph{smooth} if every 
vertex $v\in V$ is smooth.
\end{itemize}
\end{definition}

Let $\mathcal{D}_\T$ denote the \emph{space of Delzant 
polytopes in $\mathbb{R}^n$}, where $n=\dim \T$. 
We construct a topology on $\mathcal{D}_\T$, coming from a 
metric. 

Recall that the \emph{symmetric difference} of two subsets 
$A,B \subset \mathbb{R}^n$ is
$$
A\Delta B:= (A\smallsetminus B) \cup (B\smallsetminus A).
$$ 

Let $\mathfrak{B}(\R^n)$ be the $\sigma$-algebra of Borel 
sets of $\R^n$, and let $\lambda\colon \mathfrak{B}(\R^n)
\to \R_{\geq 0}\cup \{\infty\}$ be the Lebesgue measure
on $\R^n$. 

\begin{definition}
Let $\mathfrak{B}'(\R^n)\subset \mathfrak{B}(\R^n)$
be the Borel sets with finite Lebesgue measure. Define 
$$
d\colon \mathfrak{B}'(\R^n)\times \mathfrak{B}'(\R^n)\to 
\R_{\geq 0}
$$ 
by
\begin{equation}\label{distance L1}
d(A,B):=\lambda(A\Delta B)=
\int_{\R^n} \left|\chi_A-\chi_B\right|\,{\rm d}\lambda 
= \left\|\chi_A-\chi_B\right\|_{{\rm L}^1},
\end{equation}
where $\chi_C\colon \R^n\to \R$ denotes the characteristic 
function of $C\in \mathfrak{B}'(\R^n)$. 
\end{definition}

Note that $d$ is symmetric and satisfies the triangle 
inequality, since 
$$
A\Delta C\subset (A\Delta B) \cup (B\Delta C).
$$ 
However, in this space,  $d(A,B)=0$ does not necessarily imply
that $A=B$. We introduce in $\mathfrak{B}'(\R^n)$ the equivalence relation $\sim$,
where 
$$
A\sim B\quad\text{if and only if}\quad\lambda(A\Delta B)=0.
$$ 
Then the induced map, also denoted by $d\colon 
(\mathfrak{B}'(\R^n)/\sim)\times (\mathfrak{B}'(\R^n)/\sim)
\to \R_{\geq 0}$, is a metric (associated to the ${\rm L}^1$ 
norm).

Since $A,B\in \mathcal{D}_\T\subset \mathfrak{B}'(\R^n)$ and 
$A\sim B$ implies $A=B$, it follows that 
$$
(\mathcal{D}_\T/\sim)\,=\mathcal{D}_\T
$$ 
and thus the restriction of $d$ to $\mathcal{D}_\T$ is a 
metric. Hence $(\mathcal{D}_\T,d)$ is a metric space, endowed 
with the topology induced by $d$.

\subsection{Symplectic toric manifolds}\label{stm}

We review below the ingredients from the theory of symplectic 
toric manifolds which we need for this paper, in particular 
the Delzant classification theorem. We follow the conventions 
in Section \ref{conventions}.

A \emph{symplectic manifold} $(M,\omega)$ is a pair 
consisting of a smooth manifold $M$ and a \emph{symplectic 
form} $\omega$, i.e., a non-degenerate closed 2-form on $M$. 
Suppose that the $n$-dimensional torus $\T$ acts on 
$(M,\, \omega)$  symplectically (i.e., by diffeomorphisms 
which preserve the symplectic form). The action 
$\T \times M \to M$ of $\T$ on $M$ is denoted by $(t,m) 
\mapsto t \cdot m$.

A vector $X$ in the Lie algebra $\mathfrak{t}$ generates a 
smooth vector field $X_M$ on $M$, called the 
\emph{infinitesimal generator}, defined
by 
$$
X_M(m):= \left.\frac{{\rm d}}{{\rm d}t}\right|_{t=0}
{\rm exp}(tX)\cdot m,
$$ 
where ${\rm exp} \colon \mathfrak{t} \to \T$ is the 
exponential map of Lie theory and $m \in M$. We write the contraction $1$-form as
$\iota_{X_M} \omega : = \omega(X_M, \cdot) \in \Omega^1(M)$. 

Let $\left\langle \cdot , \cdot
\right\rangle : \mathfrak{t}^\ast \times \mathfrak{t}
\rightarrow \mathbb{R}$ be the duality pairing.
The  $\T$-action on $(M,\omega)$  is said to be 
\emph{Hamiltonian} if there exists a smooth
$\mathbb{T}$-invariant map $\mu\colon M \to \mathfrak{t}^*$, 
called the \emph{momentum map}, such that for
all $X \in \mathfrak{t}$ we have 
\begin{eqnarray}
\iota_{X_M}
\omega ={\rm d} \langle \mu,
X \rangle. \label{oo}
\end{eqnarray}

As defined in Section \ref{sec:intro}, a \emph{symplectic 
toric manifold} $(M, \om, \T, \mu)$ is a symplectic compact 
connected manifold $(M, \om)$ of dimension $2n$ endowed with 
an effective (i.e., the intersection of all isotropy
subgroups is the identity) Hamiltonian action of an $n$-
dimensional torus $\T$ admitting a momentum map $\mu: M \to 
\mathfrak{t}^*$.  With the conventions of
Section \ref{conventions}, the map $\mu: M \to \mathfrak{t}^*$ 
gives rise (in a non\--canonical way) to a map
$M \to\mathfrak{t}^* \to \mathbb{R}^n$ which, for simplicity, 
is also denoted by $\mu \colon M \to \mathbb{R}^n$.

\begin{definition} \label{def:iso}
Let $(M,\omega,\T,\mu)$ and $(M',\omega',\T,\mu')$ be 
symplectic toric manifolds, with effective symplectic 
actions $\rho\colon \T\to \symp (M,\omega)$ and $\rho'\colon 
\T\to \symp (M',\omega')$. We say that $(M,\omega,\T,\mu)$ 
and $(M',\omega',\T,\mu')$ are \emph{isomorphic}\footnote{In 
the literature, these manifolds are usually called 
\emph{equivariantly symplectomorphic}. However, the same name 
is sometimes also used for the notion in Definition 
\ref{weaker}, and so we use different names to distinguish the 
two.} if there exists an equivariant symplectomorphism 
$\varphi\colon M\to M'$ such that $\mu'\circ \varphi=\mu$.

We denote by $\Mo$ the moduli space of $2n$-dimensional 
isomorphic toric manifolds.
\end{definition}

The following is an influential theorem by T. Delzant 
(\cite{De1988}).

\begin{theorem}[Delzant's Theorem] \label{theo:delzant} 
There is a one-to-one correspondence between isomorphism 
classes of symplectic toric manifolds and Delzant polytopes, 
given by:
\begin{equation}\label{iso}
\left[(M,\omega,\T,\mu) \right]\ni \Mo \longmapsto
\mu(M) \in \Dt.
\end{equation}
\end{theorem}

As a consequence of this bijection, we can endow $\Mo$ with 
the pullback metric. 
\begin{definition}\label{metric on Mo}
Let $\mathbf{M}_1 = 
(M_1,\omega_1,\T,\mu_1)$ and $\mathbf{M}_2=
(M_2,\omega_2,\T,\mu_2)$ be two symplectic toric manifolds. 
We define
$d_{\T}(\mathbf{M}_1,\mathbf{M}_2)$ to be the Lebesgue measure of the symmetric difference
of $\mu_1(M_1)$ and $\mu_2(M_2)$.
\end{definition} 

\begin{remark}\label{DH}
Note that the metric $d_{\T}$ defined above 
is related to the \emph{Duistermaat-Heckman measure} 
(\cite{DuHe1982}). Indeed, for a symplectic toric manifold 
$M$ with momentum map $\mu$, the induced Duistermaat-Heckman 
measure of a Borel set $U\subset \R^n\simeq \mathfrak{t}^*$ 
is given by 
$$
m_{{\rm DH}}(U)=\lambda(U\cap \mu(M)).
$$
\end{remark}

\begin{remark} \label{De}
T. Delzant \cite[Section 5]{De1988} observed that a Delzant 
polytope gives rise to a fan (``\'eventail" in French), 
and that the symplectic toric manifold with associated 
Delzant polytope $\Delta$ is $\T$\--equivariantly 
diffeomorphic to the \emph{toric variety} 
defined by the fan. 

The toric variety is an $n$-dimensional complex analytic 
manifold, and the action of the real torus $\T$ on it
has an extension to a complex analytic action on  
the complexification $\T_{\mathbb{C}}$ of $\T$. 
\end{remark}

\begin{remark}
In dimension $4$, there is another class of integrable systems 
which is classified: those called semitoric 
\cite{PeVN2009,PeVN2011}. The classification of almost 
toric systems is begun in \cite{PeRaVN2012}.
\end{remark}

Now we introduce a weaker notion of equivalence between toric 
manifolds, following \cite{KKP}.

\begin{definition}\label{weaker}
Two symplectic toric manifolds $(M,\omega,\T,\mu)$ and 
$(M',\omega',\T,\mu')$ are \emph{weakly isomorphic} if there 
exists an automorphism of the torus $h\colon \T\to \T$ and 
an $h$-equivariant symplectomorphism $\varphi\colon M\to M'$, 
i.e., the following diagram commutes:
\begin{equation}\label{comm action2}
\xymatrix{
\T\times M \;\; \ar[r]^{\;\;\rho^*} \ar[d]_{(h,\varphi)} & M 
\ar[d]_{\varphi} \\
\T \times M' \;\;\ar[r]^{\;\;\rho'^*} &  M', } 
\end{equation}
where $\rho^\ast(x,m): = \rho(x)(m)$ for all $x\in\mathbb{T}$, 
$m \in M$, and similarly for $\rho'^*$.
We denote by $\Mt$ the moduli space of weakly isomorphic 
$2n$-dimensional toric manifolds.
\end{definition}

It is easy to see that two weakly isomorphic toric manifolds 
are isomorphic if $h$ is the identity and 
$\mu' \circ \varphi =\mu$.

Recall that the automorphism group of the torus $\T=\R^n/(2\pi\Z)^n$ 
is given by $\operatorname{GL}(n,\Z)$; thus the automorphism 
$h$ is represented by a matrix $A\in \operatorname{GL}(n,\Z)$.
Let $\agl(n,\Z)$ be the group of affine transformations of 
$\R^n$ given by
\begin{equation}
\label{aff_tr}
x\mapsto Ax+c,
\end{equation} 
where $A\in \operatorname{GL}(n,\Z)$ and $x, c\in \R^n$. Two
sets are \emph{$\agl(n,\Z)$-congruent} if one is the image
of the other by an affine transformation \eqref{aff_tr}. 
We have the following result.

\begin{prop}{\rm (\cite[Proposition 2.3 (2)]{KKP})}
\label{K et al}
Two symplectic toric manifolds $(M,\omega,\T,\mu)$ and 
$(M',\omega',\T,\mu')$ are weakly isomorphic if and only if 
their momentum map images are $\agl(n,\Z)$-congruent.
\end{prop}

Thus, if we define $\Dtt$ to be the moduli space of 
$\agl(n,\Z)$-equivalent (or, congruent) Delzant polytopes, 
$$\Dtt:=\Dt/\agl(n,\Z),$$ 
Proposition \ref{K et al} implies that  the isomorphism in 
\eqref{iso} descends to an isomorphism
\begin{equation}\label{iso moduli}
\Mt \longrightarrow \Dtt.
\end{equation}
Let $\pi\colon \Dt\to \Dtt$ be the projection map; we endow 
$\Dtt$ with the quotient topology $\widetilde{\delta}$.

We endow $\Mt$ with the topology $\widetilde{\nu}$ induced by 
the isomorphism \eqref{iso moduli}.

\section{Connectedness of the space  
$(\Dtt,\widetilde{\delta})$} 
\label{sec:conn}

\subsection{Classification of Delzant polygons in 
$\mathbb{R}^2$}

We introduce now the definitions of rational length and corner 
chopping of size $\varepsilon$, which will be used in the 
classification of the Delzant polytopes in $\R^2$.

\begin{definition}(following \cite[2.4 and 2.11]{KKP})
\begin{itemize}
\item[(i)]
The \emph{rational length} of an interval $I$ of rational 
slope in $\R^n$ is the unique number $l=\len(I)$ such that 
the interval is $\operatorname{AGL}(n,\Z)$-congruent
to an interval of length $l$ on a coordinate axis. 
\item[(ii)]
Let $\Delta$ be a Delzant polytope in $\R^n$ and $v$ a vertex 
of $\Delta$. Let 
$$
\{v+t u_i\mid 0\leq t\leq \ell_i\}
$$ 
be the set of edges emanating from $v$, where the 
$u_1,\ldots,u_n$ generate the lattice $\Z^n$ and
$\ell_i=\len(u_i)$ (as defined above in (i)).

For $\varepsilon>0$ smaller than the $\ell_i$'s, 
the \emph{corner chopping of size} $\varepsilon$ of $\Delta$ 
at $v$ is the polytope $\Delta'$ obtained from $\Delta$ by 
intersecting it with the half space  
$$
\{v+t_1u_1+\cdots +t_nu_n \mid t_1+\cdots+t_n\geq 
\varepsilon\}.
$$
\end{itemize}
\end{definition}

\begin{figure}[h!]
\begin{center}
\epsfxsize=\textwidth
\leavevmode
\psfrag{d}{$\varepsilon$}
\includegraphics[width=1.6in]{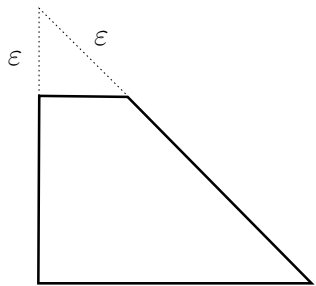}
\end{center}
\caption{A corner chopping of size $\varepsilon$.}
\label{moving edge 1}
\end{figure}

In $\mathbb{R}^2$, all Delzant polygons can be obtained by a 
recursive recipe, which can be found in 
\cite[Lemma 2.16]{KKP} and is recalled below.

\begin{lemma} {\rm (}see {\rm \cite{Fu1993}, Sec. 2.5 and Notes to Chapter 2)}
\label{fulton}
The following hold:
\begin{enumerate}
\item \label{f1}
Any Delzant polygon $\Delta \subset \mathbb{R}^2$ with three edges is $\operatorname{AGL}(2,\mathbb{Z})$-congruent to the Delzant triangle $\Delta_\lambda$ for a unique $\lambda > 0$ (Example \ref{ex:Delzant}).
\item \label{f2}
For any Delzant polygon $\Delta \subset \mathbb{R}^2$  with $4+s$ edges, where $s$ is a non-negative integer, there exist positive numbers $a \geq b > 0$, an integer $0 \leq k \leq 2a/b$, and positive numbers $\varepsilon_1, \ldots,\varepsilon_s$, such that $\Delta$ is $\operatorname{AGL}(2,\mathbb{Z})$-congruent to a Delzant polygon that is obtained from the Hirzebruch trapezoid $H_{a,b,k}$ (see Example \ref{ex:Delzant}) by a sequence of corner choppings of sizes $\varepsilon_1, \ldots, \varepsilon_s$.
\end{enumerate}
\end{lemma}

\begin{figure}[h!]
\begin{center}
\epsfxsize=\textwidth
\psfrag{l}{$\lambda$}
\psfrag{a}{$a$}
\psfrag{b}{$b$}
\psfrag{s}{$\mbox{slope}\,=-\frac{1}{k}$}
\leavevmode
\includegraphics[width=3.7in]{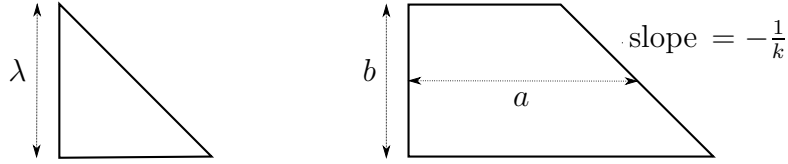}
\end{center}
\caption{The Delzant triangle $\Delta_{\lambda}$ and the Hirzebruch trapezoid $H_{a,b,k}$.}
\label{fig:polygons}
\end{figure}

\begin{example} \label{ex:Delzant}
Figure \ref{fig:polygons} shows the Delzant triangle  
$\Delta_{\lambda}$ and the Hirzebruch trapezoid $H_{a,b,k}$.
The \emph{Delzant triangle},
$$ 
\Delta_\lambda := \left\{ \left.(x_1,x_2) \in \mathbb{R}^2 
\ \right| \
 x_1 \geq 0 ,\; x_2 \geq 0 ,\,\,\, x_1 + x_2 \leq \lambda 
 \right\}, 
$$
is the momentum map image of the standard $\T^2$ action on $
\mathbb{CP}^2$ endowed
with the Fubini-Study symplectic form multiplied by $\lambda$. 

The \emph{Hirzebruch trapezoid}, 
$$ 
H_{a,b,k} := \left\{ (x_1,x_2)\in\mathbb{R}^2 \ \left| \
   -\frac{b}{2} \leq x_2 \leq \frac{b}{2}, \right.\,\,\, 0 \leq x_1 
   \leq a - kx_2 \right\}, 
$$
is the momentum map image of the standard toric action on a 
Hirzebruch surface.

Here, $b$ is the height of the trapezoid, 
$a$ is its average width, and $k\geq 0$ is a non-negative 
integer such that the right edge has slope $-1/k$ or is 
vertical if $k=0$. Moreover $a$ and $b$ have to satisfy
$a \geq b$ and $a - k \frac{b}{2} >  0$.
\end{example}

\subsection{Proof of Theorem \ref{key} (a)} 
Recall that $(\mathcal{D}_\T,d)$ is the space of Delzant 
polytopes in $\R^2$ together with the distance function given 
by the area of the symmetric difference and that $\Dtt$ is 
the quotient space $\Dt/{\rm AGL}(2,\Z)$ with the quotient 
topology $\widetilde{\delta}$, induced by the quotient map 
$\pi\colon \Dt\to\Dtt$.

In order to prove Theorem \ref{key} (a), given the isomorphism 
\eqref{iso moduli}, it is enough to prove the following 
statement.

\begin{theorem}
The space $(\Dtt,\widetilde{\delta})$ is path-connected.
\end{theorem}

\begin{proof}
Let $\mathcal{S}\subset \Dt$ be the subset that contains all 
Delzant triangles $\Delta_{\lambda}$ for $\la \in\R^+$,
all Hirzebruch trapezoids $H_{a,b,k}$ with $a,b\in \R$ such 
that $a\geq b>0$ and $k\geq 2a/b$ a non-negative integer, and 
also all Delzant polygons obtained from Hirzebruch trapezoids
by a sequence of corner choppings (cf. Lemma \ref{fulton}).  
Endow $\mathcal{S}\subset (\Dt,d)$ with the subspace topology.
We will prove that $\mathcal{S}$ is path-connected in $\Dt$.
In fact, by Lemma \ref{fulton}, we know that every element of 
$\Dtt$ has a representative in $\mathcal{S}$.
Hence, given $[P_0],[P_1]\in \Dtt$ with representatives 
$P_0,P_1\in \mathcal{S}$, if we prove that there exists a 
continuous path $\gamma\colon [0,1]\to \mathcal{S}$ such that
$\gamma(0)=P_0$ and $\gamma(1)=P_1$, by the continuity of 
$\pi|_{\mathcal{S}}\colon \mathcal{S}\to \Dtt$, it follows 
that there exists a continuous path
$\pi\circ \gamma\colon [0,1]\to\Dtt$ connecting $[P_0]$ 
to $[P_1]$, thus proving that $\Dtt$ is path-connected.

First of all, note that the intuitive paths from a Delzant 
polygon $P$ to a translation of $P$ or a scaling of
$P$ by a positive factor (or a composition of the two) are 
clearly continuous with respect to the topology
induced by $d$.  Furthermore, if $P\in \mathcal{S}$ then 
the entire path also lies in $\mathcal{S}$. The same holds 
when moving an edge parallel to itself without changing the 
total number of edges (see Figure \ref{moving edge}).

\begin{figure}[h!]
\begin{center}
\epsfxsize=\textwidth
\leavevmode
\includegraphics[width=2.2in]{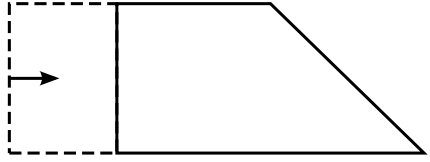}
\end{center}
\caption{Moving an edge parallel to itself.}
\label{moving edge}
\end{figure}

In particular, all Delzant triangles $\Delta_{\la}$ are 
connected by a path in $\mathcal{S}$ and so are the 
Hirzebruch trapezoids $H_{a,b,k}$, for each fixed $k$.

Secondly, if $P_{\varepsilon}$ is obtained from $P$ by a 
corner chopping of size $\varepsilon$ at a vertex $v\in P$, 
then $P$ and $P_{\varepsilon}$ are connected by a continuous 
path in $\Dt$; the path is simply given by 
$$\gamma\colon [0,1]\to \Dt \text{ with }\gamma(t):=P_{t\varepsilon}.$$
Thus any Delzant polygon obtained 
from $H_{a,b,k}$ by a sequence of  corner choppings
is connected to $H_{a,b,k}$ via a path in $\mathcal{S}$.

Third, we next show that for $k\geq 0$ there is a continuous 
path between the Hirzebruch trapezoid $H_{a,b,k}$ and the
Hirzebruch trapezoid $H_{a,b,k+1}$, where now $0\leq k,k+1\leq 2a/b$.
The first half of the path connects $H_{a,b,k}$ to the polygon $H'_{a,b,k}$ by corner chopping at the top right vertex, and the
inverse of the
 second half 
connects $H_{a,b,k+1}$ to $H'_{a,b,k}$ by corner chopping at the bottom right vertex (cf. Figure \ref{Hs}).

\begin{figure}[h!]
\begin{center}
\epsfxsize=\textwidth
\leavevmode
\psfrag{a}{$a$}
\psfrag{b}{$b$}
\psfrag{s1}{$\mbox{slope}\,=\;\;-\frac{1}{k+1}$}
\psfrag{s2}{$\mbox{slope}\,=\;\;-\frac{1}{k}$}
\includegraphics[width=5.3in]{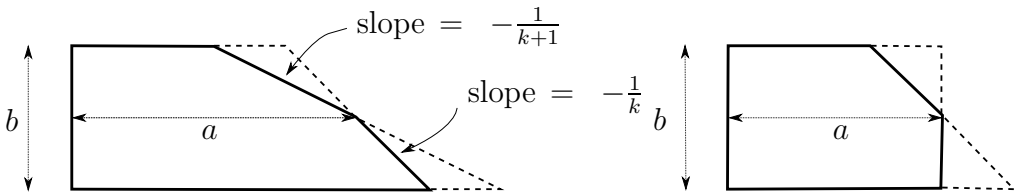}
\end{center}
\caption{$H'_{a,b,k}$ for $k\geq 1$ and $H'_{a,b,0}$.}
\label{Hs}
\end{figure}

Note that $H'_{a,b,k}$ is still a Delzant polygon; it suffices 
to check smoothness at the new vertex, and indeed 
$$
\det\left[
\begin{array}{cc}
-(k+1) & k\\
1 & -1
\end{array}
\right]=1.
$$ 
Combining with previous observations, we conclude that all 
Hirzebruch trapezoids with $k\geq 0$ lie
in the same path-connected component of $\mathcal{S}$.

Finally, in order to conclude that $\mathcal{S}$ is 
path-connected, it remains to check that there exists a 
continuous path between, for example, $H_{\la,\la,0}$ and 
$\Delta_\la$. Let $\gamma\colon [0,1]\to \mathcal{S}$
be the path such that:
\begin{itemize}
\item[(i)]
$\gamma(t)$ is the corner chopping of size $\la t$ at the 
top right vertex of the square $H_{\la,\la,0}$ for $0< t<1$, 
\item[(ii)]
$\gamma(0):=H_{\la,\la,0}$, and 
\item[(iii)]
$\gamma(1):=\Delta_{\la}$.
\end{itemize}
The path $\gamma$ is continuous with respect to the topology 
on $\Dt$.
$\quad \square$
\end{proof}

\section{Topology of the space $(\Dt,d)$} 
\label{sec:compl}
In this section we prove Theorem \ref{key} (b). By the 
isomorphism \eqref{iso moduli}, it suffices to study
the topological properties of $(\Dt,d)$. 

Let $(\mathfrak{B}'(\R^2)/\sim,d)$ be the metric space 
introduced in Section \ref{subsection:metric}.
\begin{prop}
The space $(\Dt,d)$ is not complete. 
\end{prop}

\begin{proof}
We prove that $(\Dt,d)$ is not complete, by
giving an example of a Cauchy sequence in $(\Dt,d)$ which 
converges in $(\mathfrak{B}'(\R^2)/\sim,d)$ whose limit is not
a smooth polytope, hence not in $\Dt$. For $k\neq 1$ consider 
the Hirzebruch surface $H_{a,b,k}$, and note that we can 
rewrite $a$ as $$a=c+\frac{bk}{2},$$ where $c$ is the length 
of the top facet. Then, the sequence 
$$
H_{\frac{c}{n}+\frac{bk}{2},b,k},\,\,\,\,\,\, n=1,2,3,\ldots
$$
is Cauchy, but its limit is a right angle triangle that is not 
Delzant (see Figure \ref{limit}). 
$\quad \square$
\end{proof}

\begin{remark}
Note that $(\Dt, d)$ is also non-compact, since compact metric 
spaces are automatically complete.
\end{remark}

\begin{figure}[h!]
\begin{center}
\epsfxsize=\textwidth
\leavevmode
\psfrag{v}{$v$}
\includegraphics[width=2in]{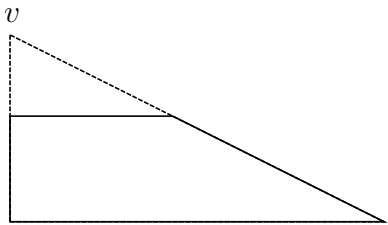}
\end{center}
\caption{The vertex $v$ is not smooth.}
\label{limit}
\end{figure}

Let $\mathcal{C}\subset \mathfrak{B}'(\R^2)$ be the space of 
convex compact subsets of $\R^2$ with positive Lebesgue 
measure. Note that if $A,B\in \mathcal{C}$ and $d(A,B)=0$ 
then $A=B$, and so $d$ is a metric on $\mathcal{C}$. 
The same observations hold for $\mathcal{P}_2$, the space of 
convex 2-dimensional polygons in $\R^2$, and $\mathcal{P}_\Q$, 
the space of rational convex 2-dimensional polygons in $\R^2$ 
(see Definition \ref{del 2}). We have the following inclusions 
of metric spaces:
$$
(\Dt,d)\subset(\mathcal{P}_\Q,d)\subset(\mathcal{P}_2,d)
\subset(\mathcal{C},d).
$$

As we will see in the proof of Theorem \ref{prop:completion}, 
all these inclusions are dense.
 
Let $\varnothing$ be the empty set, which clearly has zero 
Lebesgue measure, and define
$$
\hat{\mathcal{C}}:=\mathcal{C}\cup\varnothing.
$$ 
Thus $\hat{\mathcal{C}}\subset \mathfrak{B}'(\R^2)$ and
$(\hat{\mathcal{C}},d)$ is a metric space.
Observe that $(\mathcal{C},d)\subset (\hat{\mathcal{C}},d)$,
where the inclusion is a continuous map of metric spaces. 
With a similar argument 

As before, we have the inclusions of metric spaces:
$$
(\Dt,d)\subset(\mathcal{P}_\Q,d)\subset(\mathcal{P}_2,d)
\subset(\mathcal{C},d)\subset (\hat{\mathcal{C}},d).
$$

We will prove in Theorem \ref{C complete} that 
$(\hat{\mathcal{C}},d)$ is complete.
To do this, we need the following lemma.

\begin{lemma}\label{convex}
Let $A\in \mathfrak{B}(\R^2)$ be convex and non-bounded. 
If $\lambda(A)>0$, then $\lambda(A)=\infty$. 
\end{lemma}
\begin{proof}
Consider two points $\mathbf{p},\mathbf{q}\in A$. By convexity 
of $A$, the straight line segment 
$\ell_{\mathbf{p}\mathbf{q}}$
connecting $\mathbf{p}$ to $\mathbf{q}$ is contained in $A$.
Since $\lambda(A)>0$, there exists a point $\mathbf{r}\in A$ 
not collinear to $\mathbf{p}$ and $\mathbf{q}$. Hence the 
whole triangle $\mathbf{p}\mathbf{q}\mathbf{r}$
is contained in $A$. 

Let $C$ be the circle inscribed in the triangle 
$\mathbf{p}\mathbf{q}\mathbf{r}$, and $\{a_n\}_{n\in \N}$ 
a sequence of points in $A$ such that 
$\|a_n\|_{\R^2}\to \infty$. 
For each $a_n$, there exists a diameter $D_n$ of $C$ such 
that the triangle $T_n\subset A$ with base
$D_n$ and third vertex $a_n$ is isosceles, which guarantees 
that $\lambda(T_n)\to \infty$, and so $\lambda(A)=\infty$.
$\quad \square$
\end{proof}

\begin{theorem}\label{C complete}
$(\hat{\mathcal{C}},d)$ is a complete metric space and it is 
the completion of $(\mathcal{C},d)$.
\end{theorem}
\begin{proof}
Let $\{A_n\}_{n\in \N}$ be a Cauchy sequence in 
$\hat{\mathcal{C}}$, with $A_n$ convex. By definition
of $d$ (see \eqref{distance L1}), the sequence 
$\{\chi_{A_n}\}_{n\in \N}$ is Cauchy in ${\rm L}^1(\R^2)$.
By completeness of ${\rm L}^1(\R^2)$, there exists a function 
$f\in {\rm L}^1(\R^2)$ such that 
$$
\left\|\chi_{A_n}-f\right\|_{{\rm L}^1}\rightarrow 0.
$$
Thus, there exists a subsequence $\chi_{A_{n_k}}$ and a zero 
measure set $E\subset \R^2$ such that 
$$
\chi_{A_{n_k}}(x)\to f(x)
$$ 
for all $x\in \R^2\setminus E$. Let 
$$
A=\{x\in \R^2\setminus E\mid f(x)=1\};
$$ 
from the definitions it follows immediately that
$\chi_{A_{n_k}}(x)\to \chi_{A}(x)$ for all 
$x\in \R^2\setminus E$ and 
$$
\left\|\chi_{A_n}-\chi_A\right\|_{{\rm L}^1}\to 0.
$$

It is easy to see that $\lambda(A)< \infty$, thus 
$A\in \mathfrak{B}'(\R^2)$. If $\lambda(A)=0$ we can take $A$ 
to be $\varnothing$, which belongs to $\hat{\mathcal{C}}$. 
Let us now assume that $\lambda(A)>0$; we prove that
$A$ is almost everywhere equal to a convex compact subset 
of $\R^2$. Let $A'$ be the convex hull of $A$. Then, for 
any $\mathbf{p}\in A'$ there exists $\mathbf{q},
\mathbf{r}\in A$ such that
$$
\mathbf{p}=t\mathbf{q}+(1-t)\mathbf{r}
$$ 
for some $t\in [0,1]$. Since $\mathbf{q},\mathbf{r}\in A$, 
there exists $N\in \N$ such that for all $l> N$
$$
\chi_{A_{n_l}}(\mathbf{q})=\chi_{A_{n_l}}(\mathbf{r})=1\;,
$$ 
that is, $\mathbf{q},\mathbf{r}\in A_{n_l}$ for all $l> N$.
Since $A_n$ is convex for all $n$, this means that 
$\mathbf{p}\in A_{n_l}$ for all $l>N$,
which implies that 
$$
\chi_{A_{n_k}}(x)\to \chi_{A'}(x)
$$ 
almost everywhere in $A'\cup (\R^2\setminus A)=\R^2$.
Hence 
$$
\lambda(A'\setminus A)=0.
$$ 
Now it is sufficient to observe that, since $A'$ is convex, 
its boundary $\partial A'$ has Lebesgue measure zero 
(see \cite{La}), and since $\overline{A'}$
is also convex with $\lambda(\overline{A'})=
\lambda(A)<\infty$, by Lemma  \ref{convex} it is bounded, 
hence compact.
$\quad \square$
\end{proof}

Before investigating what the completion of $(\Dt,d)$ is, we
first prove an auxiliary result, the content of which is 
related to resolving singularities on toric varieties
(see Remark \ref{De} and \cite{Cox}).
Recall that a vector $u\in \Z^2$ is called \emph{primitive} 
if, whenever $u=kv$ for some $k\in \Z$ and $v\in \Z^2$,
then $k=\pm 1$.

\begin{lemma}
\label{lemma:resolving}
Let $\Delta$ be a simple rational polytope  in $\R^2$ that 
fails to be smooth only at one vertex $\mathbf{p}$: the 
primitive vectors $u,v\in\Z^2$ which direct the edges at 
$\mathbf{p}$ do not form a $\Z$-basis of $\Z^2$. Then there 
is a simple rational polytope $\widetilde{\Delta}$ with at 
most $|{\rm det}\left[u\,   v\right]|-1$ more edges than 
$\Delta$ that is a smooth polytope and is equal to $\Delta$ 
except in a neighborhood of $\mathbf{p}$.
\end{lemma}

\begin{proof} 
Let $u=(a,b)$ and $v=(c,d)$. We claim that there is always a 
matrix $A\in {\rm GL}(2,\Z)$ and $(\alpha_0,\alpha_1)\in\Z^2$, 
a primitive vector, such that 
$$
A \,\left[ \begin{array}{cc}
a & c \\
b & d \end{array} \right]=
\left[ \begin{array}{cc}
1 & \alpha_0 \\
0 & \alpha_1 \end{array} \right].
$$
If we set $\alpha_1=|\text{det}\left[u  \, v\right]|$, the claim is equivalent to being able to solve the following for 
$\alpha_0$:
$$
\left\{\begin{array}{cc}
a \,\alpha_0 \equiv c&\text{ mod }\alpha_1\\
b \,\alpha_0\equiv d&\text{ mod }\alpha_1
\end{array}\right. .
$$
Note also that $\alpha_0\not\equiv 0$ mod $\alpha_1$, for 
otherwise it would contradict that $(\alpha_0,\alpha_1)$ is 
primitive.

The primitive vectors directing the edges of the polygon 
$A(\Delta)$ at the non-smooth vertex $A(\mathbf{p})$ are 
$(1,0)$ and $(\alpha_0,\alpha_1)$. We can additionally do a 
shear transformation via a matrix of the form 
$$
S_1=\left[\begin{array}{cc}
1 & k \\
0 & 1
\end{array}\right]
$$
and obtain a ${\rm GL}(2,\Z)$-equivalent polygon 
$S_1 A (\Delta)$ for which the non-smooth vertex 
$S_1 A (\mathbf{p})$ has edge directing vectors $(1,0)$ 
and $(\alpha_2,\alpha_1)$ where $0<\alpha_2<\alpha_1$.

We now create a new vertical edge on the polygon 
$S_1 A (\Delta)$ as close to the vertex $S_1 A (\mathbf{p})$ 
as desired, thereby eliminating that vertex. Call this new 
polygon $\Delta_1$. Of the two vertices at the endpoints of 
this new edge, one is clearly smooth: the one with edge 
directing vectors $(1,0)$ and $(0,1)$. The other vertex has 
edge directing vectors $(0,-1)$ and $(\alpha_2,\alpha_1)$ and 
is smooth if and only if 
$$
\text{det}\left[\begin{matrix}
0&\alpha_2\\ -1&\alpha_1
\end{matrix} \right]=\alpha_2=1.
$$ 
If this second vertex is smooth, the desired polygon 
$\widetilde{\Delta}$ is 
$$
A^{-1} S_1^{-1}(\Delta_1).
$$ 
Otherwise, the process continues: let $B$ be the matrix
$$B=\left[ \begin{array}{cc}
0 & -1 \\
1 & 0 \end{array} \right].
$$
Then the polygon $B (\Delta_1)$ is smooth except at one 
vertex $\mathbf{p}_1$, which has edge directing vectors 
$B((0,-1))=(1, 0)$ and $B((\alpha_2,\alpha_1))=
(-\alpha_1,\alpha_2)$. As before, we can apply a shear 
transformation to obtain $S_2B(\Delta_1)$ such that the edge 
directing vectors at the non-smooth vertex 
$S_2B(\mathbf{p}_1)$ are of the form $(1,0)$ and 
$(\alpha_3,\alpha_2)$ with $0<\alpha_3<\alpha_2$. Proceeding 
exactly as before, we create a new vertical edge on the 
polygon $S_2B(\Delta_1)$ as close to the vertex 
$S_2B(\mathbf{p}_1)$ as desired, thereby eliminating it. Call 
this new polygon $\Delta_2$. Of the two vertices at the 
endpoints of the new edge, the one with edge directing vectors 
$(0,1)$ and $(1,0)$ is clearly smooth, the other one has edge 
directing vectors $(0,-1)$ and $(\alpha_3,\alpha_2)$ and is 
smooth if and only if $\alpha_3= 1$. If that is so, the 
desired polygon $\widetilde{\Delta}$ is 
$$
A^{-1}S_1^{-1}B^{-1}S_2^{-1}(\Delta_2),
$$ 
otherwise we repeat the process. 

Because $\alpha_1, \alpha_2, \alpha_3, \ldots$ is a strictly 
decreasing sequence of non-negative integers, it will reach 
$1$ in at most $\alpha_1-1$ steps, thus terminating the 
process and producing $\widetilde{\Delta}$ with at most 
$\alpha_1-1$ more edges than $\Delta$. This proves the 
statement in the lemma.
$\quad \square$
\end{proof}

\begin{remark}\label{remark:resolving}
Note that the process described in the proof of Lemma 
\ref{lemma:resolving} does not rely on $\Delta$ being smooth 
at all vertices other than $\mathbf{p}$. In fact, the result 
holds for any simple rational non-smooth polytope $\Delta$, 
with any number of non-smooth vertices, except that the number 
of extra edges will be 
$$
\sum_{\mathbf{p}_i}\left(|\text{det}\left[u_i\,v_i\right] |-1 
\right),
$$ 
where the $\mathbf{p}_i$s are the non-smooth vertices of 
$\Delta$. The new polytope $\widetilde{\Delta}$ is equal to 
$\Delta$ except in neighborhoods of the vertices 
$\mathbf{p}_i$ that can be made as small as desired.
\end{remark}

Now we are ready to prove the main theorem of this section.

\begin{theorem}
\label{prop:completion}
The completion of $(\Dt,d)$ is $(\hat{\mathcal{C}},d)$.
\end{theorem}

\begin{proof} 
Recall the metric space inclusions 
$$
(\Dt,d)\subset(\mathcal{P}_\Q,d)\subset(\mathcal{P}_2,d)
\subset(\mathcal{C},d).
$$
We shall prove that the completion of $(\Dt,d)$ contains 
$(\mathcal{P}_\Q,d)$, that the completion of 
$(\mathcal{P}_\Q,d)$ contains $(\mathcal{P}_2,d)$, and that 
the completion of $(\mathcal{P}_2,d)$ contains 
$(\mathcal{C},d)$. Then by Theorem \ref{C complete} the 
conclusion follows.
\smallskip

\noindent \textbf{A.)} \textit{The completion of $(\Dt,d)$ 
contains $(\mathcal{P}_\Q,d)$.} Because $d$ is a metric on 
$\mathcal{P}_\Q$ that coincides 
with the given metric on $\Dt$, in order to prove that the 
completion of $\Dt$ contains $\mathcal{P}_\Q$ it suffices to 
show that for each $\Delta\in\mathcal{P}_\Q$ there exists a 
polygon in $\Dt$ as close to $\Delta$ as desired, relative to 
the metric $d$. Lemma \ref{lemma:resolving} and, in 
particular Remark \ref{remark:resolving}, guarantee that this 
is so.
\smallskip

\noindent \textbf{B.)} \textit{The completion of
$(\mathcal{P}_\Q,d)$ contains $(\mathcal{P}_2,d)$.}
Because $d$ is a metric on $\mathcal{P}_2$ that coincides with 
the given metric on $\mathcal{P}_\Q$, in order to prove that 
the completion of $(\mathcal{P}_\Q,d)$ contains 
$(\mathcal{P}_2,d)$ it suffices to show that for each 
$\Delta\in\mathcal{P}_2$ there exists a polygon 
$\Delta_\Q\in\mathcal{P}_\Q$ as close to $\Delta$ as desired, 
relative to the metric $d$. This rational polygon $\Delta_\Q$ 
can be obtained by approximating the irrational slopes of the 
edges of $\Delta$ by rational numbers and choosing for 
directing vectors of the edges the corresponding lattice 
vectors, and also by changing the vertex points accordingly. 
This way, we can make the symmetric difference between the 
original polygon $\Delta$ and the rational polygon $\Delta_\Q$ 
be contained in a $\varepsilon$-ball of the edges of $\Delta$, 
the area of which can be made as small as needed by making $
\varepsilon$ small enough.
\smallskip

\noindent \textbf{C.)} \textit{The completion of 
$(\mathcal{P}_2,d)$ contains $(\mathcal{C},d)$.}
Because $d$ is a metric on $\mathcal{C}$ that coincides with 
the given metric on $\mathcal{P}_2$, in order to prove that 
the completion of $(\mathcal{P}_2,d)$ contains 
$(\mathcal{C},d)$ it suffices to show that for each 
$C\in\mathcal{C}$ there exists a polygon 
$\Delta_2\in\mathcal{P}_2$ as close to $C$ as desired, 
relative to the metric $d$. In order to do this, observe that 
given a compact convex set $C\in \mathcal{C}$ and 
$\varepsilon >0$, there exists a collection of disjoint 
rectangles $\{[a_i,b_i)\times [c_i,d_i)\}_{i=1}^N$ contained 
in $C$ such that 
$$
\left\|\chi_C-\sum_{i=1}^N\chi_{[a_i,b_i)\times [c_i,d_i)}
\right\|_{
{\rm L}^1}<\varepsilon.
$$
Let $\Delta_2$ be the convex hull of 
$\bigcup_{i=1}^N[a_i,b_i]\times [c_i,d_i]$.
Since $C$ is convex, we have 
$$
\bigcup_{i=1}^n [a_i,b_i)\times [c_i,d_i)\subset \Delta_2\subset C,
$$ 
and hence 
$$
\left\|\chi_{C}-\chi_{\Delta_2}\right\|_{{\rm L}^1}<\varepsilon,
$$ 
which proves the claim.
$\quad \square$
\end{proof}

\begin{prop}
The space $(\Dt,d)$ is not locally compact.
\end{prop}
\begin{proof}
To prove that $(\Dt,d)$ is not locally compact, we show that
the closure of any open ball $B_{\varepsilon}(H_{1,1,0})
\subset \Dt$ is not compact in $(\Dt,d)$.  For each fixed 
$\varepsilon$, let $\frac{\varepsilon}{2}< \delta< 
\varepsilon$ be an irrational number, and
let $Q_\delta$ be the polygon in Figure \ref{irr}. Note that 
$Q_\delta\in\mathcal{P}_2\setminus \Dt$.
By a triangle inequality argument it is easy to see that 
$$
B_{\frac{\varepsilon}{2}}(Q_\delta)\subset B_{\varepsilon}
(H_{1,1,0}).
$$
Because $(\Dt,d)$ is dense in $(\mathcal{P}_2,d)$ (see  the 
proof of Theorem \ref{prop:completion}), there exists a 
sequence of Delzant polygons
$$
\{A_n\}_{n\in \mathbb{N}}\subset 
B_{\frac{\varepsilon}{2}}(Q_\delta)
$$ 
that converges to $Q_\delta$ in $(\mathcal{P}_2,d)$.
Thus any subsequence of $\{A_n\}$ also converges to $Q_\delta$ 
in $(\mathcal{P}_2,d)$, and hence does not
converge in $(\Dt,d)$. This proves that the closure of 
$B_\varepsilon(H_{1,1,0})$ in $\Dt$ is not compact.
$\quad \square$
\end{proof}

\begin{figure}[h!]
\begin{center}
\epsfxsize=\textwidth
\leavevmode
\psfrag{d}{$\delta$}
\psfrag{1}{$1$}
\includegraphics[width=1.5in]{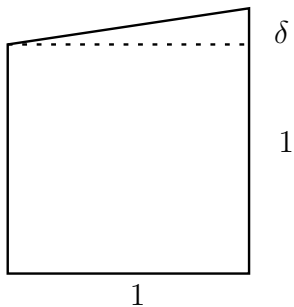}
\end{center}
\caption{The polygon $Q_\delta$.}
\label{irr}
\end{figure}

\begin{remark}\label{Hausdorff}
There is another metric commonly used on $\mathcal{C}$, namely 
the \emph{Hausdorff metric} $d_{\rm H}$ (see
\cite{BuBuIv2001}). Given two elements $A,B\in \mathcal{C}$,
we define
$$
d_{\rm H}(A,B):=\max\{\sup_{y\in B}\inf_{x\in A}\|x-y\|,\,\,\, 
\sup_{x\in A}\inf_{y\in B}\|y-x\|\}\,.
$$
As proved in \cite{ShWe}, the metrics $d$ and $d_H$ are 
equivalent on $\mathcal{C}$, and consequently all the 
topological properties proved for $(\mathcal{C},d)$ and 
$(\Dt,d)$ also hold for $(\mathcal{C},d_H)$ and $(\Dt,d_H)$. 
We chose to work with $d$ because on $\Dt$ this is related to
the Duistermaat-Heckman measure (see Remark \ref{DH}).
\end{remark}

\begin{remark}
Other moduli spaces of polygons have been studied, for 
example, in \cite{HK1,HK2} by Hausmann and Knutson and in 
\cite{KM} by Kapovich and Millson. The former focuses on the 
space of polygons in $\R^k$ with a fixed number of edges up to 
translations and positive homotheties, whereas the latter
studies the space of polygons in $\R^2$ with fixed side 
lengths up to orientation preserving isometries. 
However these different contexts completely change the flavor 
of the topological problem. 
\end{remark}

\section{Further problems} \label{sec:problems}

Theorem \ref{key} leads to further questions (not directly 
related among themselves).

\begin{problem}[Other moduli spaces]
First of all, we recall that symplectic toric manifolds
are always K\"ahler (see \cite{De1988,Gui}).
Besides the spaces we introduced, it would be
interesting to study the topological properties of the 
following:
\begin{itemize}
\item[(a)] $\Mo/\simeq$,
where $\simeq$ corresponds to rescaling the symplectic form, 
with the quotient topology. This corresponds to
considering the space $\Dt/{\rm PSL}(n;\Z)$.
\item[(b)] $\Mo/\approx$, where we also identify 
the K\"ahler manifolds
$(M,\omega)$ 
and $(M,\omega')$ if the
cohomology classes of their K\"ahler forms live in the same
connected component of the K\"ahler cone.
\end{itemize}
\end{problem}

\begin{problem}[Completeness at manifold level] 
\normalfont
This question attempts to make more explicit the relation 
between the completion at the level of polytopes with the 
completion at the
level of manifolds given in Theorem \ref{key}.
View $\Mo$ as a subset of the set of (all) integrable systems 
on symplectic $4$\--manifolds
$$
\mathcal{F}:=\{(M,\omega,F)\mid F:=(f_1,f_2) \colon M \to 
\mathbb{R}^2\}.
$$ 
The map $ v \colon \mathcal{F} \to  \mathfrak{B}'(\R^2)$ 
given by 
$$
v(M,\omega,F):=F(M)
$$ 
extends (\ref{map}). What can one say about the intersection 
$$
 V:= \hat{\mathcal{C}} \cap v(\mathcal{F})?
$$
Even though in this question $F$ is assumed smooth (because 
integrable systems $F \colon M\to \R^2$ are usually required 
to be everywhere smooth), the case when $F$ is just continuous 
on $M$ and smooth on an open dense subset of $M$ is also 
interesting. Also, it would be interesting to investigate 
under which conditions $F$ integrates to a 
$\mathbb{T}^2$-action on (an open dense subset of) $M$. 
(See, for example, \cite[Proposition 2.12]{vungoc}.)

A different but related approach to the same problem is to 
enlarge the category of objects by relaxing the smoothness 
condition on the polytope.  For example, from the work of
Lerman and Tolman in \cite{LeTo}, we know that any rational 
convex polytope is the momentum map image of a symplectic 
toric orbifold. However we do not know of a similar 
identification for generic convex compact subsets of $\R^2$. 
\end{problem}

\begin{problem}[Higher dimensional moduli spaces]
\normalfont
This paper addresses the case $2n=4$ of Problem 2.42 in 
{\rm \cite{PeVN2012}}.  We do not have results 
on the higher dimensional case $2n \ge 6$. 
\end{problem}

\begin{problem}[Continuity of packing density function] 
\label{p30}
\normalfont
Consider the maximal density function 
$$
\Omega \colon \Mo/\simeq \;\to (0,1]
$$
 which assigns to a symplectic toric
manifold its maximal density by equivariantly embedded 
symplectic balls of varying radii 
(see {\rm \cite[Definition 2.4]{PeSc2008}}).  The function 
is most interesting when considered on $\Mo/\simeq$,
where $\simeq$ corresponds to rescaling the symplectic 
form (since rescaling the symplectic form rescales the 
polytope and does not change the density).

This paper gives a quotient topology on $\Mo/\simeq$.
With respect to this topology, where is $\Omega$ continuous? 
$\Omega$ is an interesting map even if one disregards 
topology, for instance the fiber over $1$ consists of
2 points, $\mathbb{CP}^2$, and $\mathbb{CP}^1 \times 
\mathbb{CP}^1$ (proven in {\rm \cite[Theorem 1.7]{Pe2006}}) 
with scaled symplectic forms, but for any other $x \in (0,1)$ 
the fiber is uncountable 
{\rm \cite[Theorems 1.2, 1.3]{PeSc2008}}.
\end{problem}

\begin{problem}[Toric actions and symplectic forms]
\normalfont
It would be interesting to define a  torus action on the 
moduli spaces $\Mo$ or $\Mt$. Similarly for a symplectic form 
(e.g., {\rm \cite{CoKePe2012}}). 

If one has both a torus action and a symplectic form, then 
one can formulate a notion of Hamiltonian action and
momentum map, and study connectivity and convexity properties  
of the image (see for instance {\rm \cite{HaHoJeMa2006}}).
\end{problem}

\begin{problem}[Topological invariants]
\normalfont
Compute the topological invariants (fundamental group, higher 
homotopy groups, cohomology groups, etc.) of the 
path-connected space $\Mt$. 

Some preliminary questions in this direction
are: 
\begin{itemize}
\item[(a)]
find non\--trivial loop classes in $\pi_1(\Mt)$;
\item[(b)]
find non\--trivial cohomology classes in 
${\rm H}^1(\Mt,\mathbb{Z})$.
\end{itemize}
In view of the constructions of this paper, one should be 
able to compute these classes with the aid of the concrete 
description of the polytope space. 

This problem is a particular case of 
\cite[Problem 2.46]{PeVN2012}.
\end{problem}

\section{Appendix: Polytopes} \label{sec:appendix}

Let $V$ be a finite dimensional real vector space. A 
\textit{convex polytope} $S$ in $V$ is the closed
convex hull of a finite set $\{v_1, \ldots, v_n\}$, i.e., 
the  smallest convex set containing $S$ or, equivalently,
$$
\operatorname{Conv}\{v_1, \ldots, v_n\} : = 
\left\{\sum_{i=1}^n a_i v_i\;\left|\; a_i \in [0,1],
\;\sum_{i=1}^n a_i = 1 \right\}\right..
$$
The \textit{dimension} of $\operatorname{Conv}\{v_1, \ldots, 
v_n\}$ is the dimension of the vector space 
$\operatorname{span}_ \mathbb{R}\{v_1, \ldots, v_n\}$. A 
polytope is \textit{full dimensional} if its dimension equals 
the dimension of $V$.

Note that the definition implies that a convex polytope is a 
compact subset of $V$. An \textit{extreme point} of a convex 
subset $C \subseteq V$ is a point of $C$ which does not lie 
in any open line segment joining two points of $C$. Thus, a 
convex polytope is the closed convex hull of its extreme 
points (by the Krein\--Milman Theorem \cite{KeMi1940})
called \textit{vertices}. In particular, the set of vertices 
is contained in $\{v_1, \ldots, v_n\}$. Clearly, there are 
infinitely many descriptions of the same polytope as a closed 
convex hull of a finite set of points. However, the 
description of a polytope as the convex hull of its
vertices is minimal and unique.

There is another description of convex polytopes in terms of  
intersections of half-spaces. Let $V^\ast$ be the dual of $V$ 
and denote by $\left\langle\,, \right\rangle:V ^\ast\times V
\rightarrow\mathbb{R}$ the natural non-degenerate duality 
pairing. The \textit{positive (negative) half-space} defined 
by $\alpha\in V^\ast$ and $a \in \mathbb{R}$ is defined by
$$ 
V_{\alpha, a}^{\pm}: =\left\{v\in V \,\left|\,\left\langle
\alpha,v\right\rangle\gtreqless a\right\}\right.. 
$$
Traditionally, in the theory of convex polytopes, the half 
spaces are chosen to be of the form $V_{\alpha, a}^{-}$. With 
these definitions, a convex polytope is given as a finite 
intersection of half-spaces. As for the convex hull 
representation, there are infinitely many representations of 
the same convex polytope as a finite intersection of 
half-spaces, but unlike it, a distinguished one that is 
minimal exists only for full dimensional polytopes, 
 we will describe it in the next paragraph.

A \textit{face} of a convex polytope is an intersection with 
a half-space satisfying the following condition: the boundary 
of the half-space does not contain any interior point of the 
polytope. Thus the faces of a convex polytope are themselves 
polytopes (and hence compact sets).  Let $m$ be the dimension
of a convex polytope. Then the whole polytope is 
the unique $m$-dimensional face, or \textit{body}, the
$(m-1)$-dimensional faces are called \textit{facets}, the 
$1$-dimensional faces are the \textit{edges}, and the 
$0$-dimensional faces are the vertices of the polytope. If 
the convex polytope is full-dimensional, its minimal 
and unique description as an intersection of half-spaces is 
given when the boundary of those half-spaces contain 
the facets.

\bibliographystyle{new}

\end{document}